\documentclass[a4paper,12pt,reqno]{amsart}
\usepackage[a4paper,margin=1.9cm,top=2.5cm,bottom=2.5cm,centering,vcentering]{geometry}
\usepackage{hyperref,graphicx}

\theoremstyle{plain}
\newtheorem{theorem}{Theorem}[section]
\newtheorem{identity}[theorem]{Identity}

\newtheorem{corollary}[theorem]{Corollary}
\theoremstyle{definition}

\newtheorem{remark}[theorem]{Remark}
\numberwithin{equation}{section}

\newcommand\mycom[2]{\genfrac{}{}{0pt}{}{#1}{#2}}

\numberwithin{equation}{section}

\begin{document}

\title{Some New and Old Gibonacci Identities}

    \author[P. J. Mahanta]{Pankaj Jyoti Mahanta}
    \address{Gonit Sora, Dhalpur, Assam 784165, India}
    \email{pankaj@gonitsora.com}
\author[M. P. Saikia]{Manjil P. Saikia}
\address{School of Mathematics, Cardiff University, Cardiff, CF24 4AG, UK}
\email{SaikiaM@cardiff.ac.uk, manjil@gonitsora.com}



\date{\today.}

\begin{abstract}
We present a different combinatorial interpretations of Lucas and Gibonacci numbers. Using these interpretations we prove several new identities, and simplify the proofs of several known identities. Some open problems are discussed towards the end of the paper.
\end{abstract}

\maketitle

\section{Introduction}

The Fibonacci and Lucas numbers, defined by the sequences
\[F_n=F_{n-1}+F_{n-2} \quad \text{and} \quad
L_n=L_{n-1}+L_{n-2},\] with initial values $F_0=0,F_1=1$ and
$L_0=2,L_1=1$ are two of the most widely studied number sequences
in all of mathematics. Both of them are special cases of the
Gibonacci sequence $\{G_n\}_{n\geq 0}$, defined by the same
recursion
\[
G_n=G_{n-1}+G_{n-2},
\]
now with free choice of the initial two values of $G_0$ and $G_1$.

Due to the rich interplay of the Fibonacci and Lucas numbers in number theory and combinatorics, there is a wide range of identities known for these two sequences. In fact, it can be shown easily that the two sequences are related by the following identity
\[
L_n=F_{n-1}+F_{n+1}.
\]
For the proof of this and several more identities we refer the reader to the books by Honsberger \cite{hons} and Koshy \cite{koshy}. The Fibonacci and Gibonacci numbers are also related by the following identity
\begin{equation}\label{fgeq}
    G_0F_{n-1}+G_1F_{n}=G_n,
\end{equation}
the proof of which can be found in Benjamin and Quinn's book \cite{bq}.

There is more than one way to combinatorially explain the Fibonacci sequence. One of these is to count the number of domino tilings of an $2\times n$ board using vertical and horizontal dominoes. Figure \ref{fig-t} shows an example of this. It is very easy to get the recursion of the Fibonacci numbers from this interpretation. Let us denote the number of domino tilings of the $2\times n$ board by $f_n$, then we see that $f_1=1$ and $f_2=2$ and from the recursion $f_n=f_{n-1}+f_{n-2}$ it now follows that $f_n=F_{n+1}$. This interpretation is essentially the same as the one given by Benjamin and Quinn \cite{bq}, and we shall use this without any comment in the sections to follow.

\begin{figure}[htb!]
\includegraphics[width=0.45\textwidth]{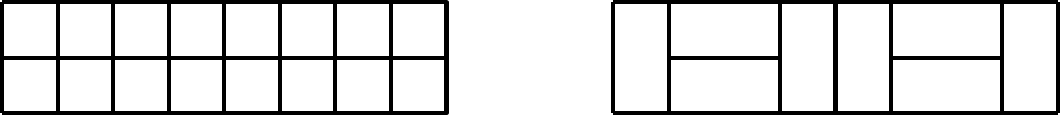}
\caption{A $2\times 8$ board on the left, the same board tiled using vertical and horizontal dominoes on the right.} \label{fig-t}
\end{figure}

There exist similar but slightly more complicated combinatorial interpretations of the Lucas sequence as well as the Gibonacci sequence. This is discussed in Chapter $2$ of Benjamin and Quinn's book \cite{bq}. The aim of this paper is to present simpler combinatorial interpretations of the Lucas and Gibonacci sequences and then use them to prove new and old identities. Our method is amenable to several techniques which have already been used successfully in the literature for proving identities related to Fibonacci numbers. Using these techniques we are able to prove several Gibonacci analogs of known Fibonacci identities, as well as other new identities involving the Gibonacci numbers.

In the following we shall always be concerned with domino tilings and hence we take the word `tiling' to mean `domino tiling' everywhere. It is also easy to see using a checkerboard representation (using two  colours) of the $2\times n$ board that every domino (both vertical or horizontal) will cover two squares of different colours. This observation will be used in several places in the following sections without mentioning it specifically.

This paper is arranged as follows: we present our interpretation of Lucas and Gibonacci numbers in Section \ref{sec:two}, then we use these interpretations to prove some new identities in Section \ref{sec:identity-new} and simplify as well as extend some known identities in Section \ref{sec:four}, in Section \ref{sec:five} we explore some further directions in which new identities can be derived but without going into too much detail, finally we end the paper with some concluding remarks in Section \ref{sec:six}.

\section{Combinatorial representations of Lucas and Gibonacci sequences}\label{sec:two}

In this short section we give new combinatorial representations of the Lucas and Gibonacci sequences.

\subsection{Combinatorial interpretation of Lucas numbers}\label{sec:cLucas}

We start with a $2 \times n$ board and then add two squares on it
marked $x$ and $y$ as shown in Figure \ref{fig1}. Let us call this
board $\mathcal{L}_n$, and we wish to tile this board with
dominoes. If a domino occupies the added squares marked $x$ and
$y$, then  the number of tilings equal $f_n$. And if a domino
occupies the added squares $x$ and the square directly below it,
then a domino is forced in the square marked $y$ and the square
below it, which in turn forces the bottom squares in columns $1$
and $2$. So in this case, total number of tilings equal $f_{n-2}$.
Hence the total number of tilings of $\mathcal{L}_n$ with dominoes
is $f_n+f_{n-2}=L_n$ for all $n\geq 1$.

\begin{figure}[htb!]
\includegraphics[width=0.39\textwidth]{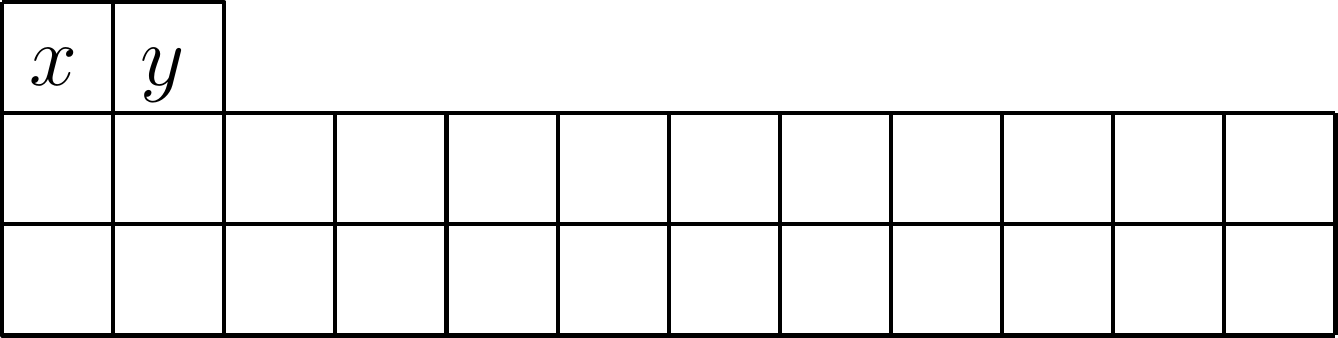}
\caption{$2\times 12$ board with two added squares on the top row.} \label{fig1}
\end{figure}

\begin{remark}\label{rem:lucas}
We can actually get the Lucas numbers recurrence directly from the
above combinatorial interpretation if instead of focusing on the
squares marked $x$ and $y$ we work from the end of the board and
see the behaviour of dominoes at the $n$-th column. The first few
values can then be easily calculated by enumerating all such
domino tilings to get $L_1=1$ (tiling a board with two squares in
the first row and one square each in the second and third row) and
$L_2=3$ (tiling a $3\times 2$ board).
\end{remark}

We now use this interpretation of Lucas numbers to prove two identities.

\begin{identity}\label{iden-lucas-1}
For all $n\geq 3$ we have
\[
L_n^2+f_n^2=f_{n-2}^2f_6+2f_{n-2}f_{n-3}f_4+f_{n-3}^2f_2.
\]
\end{identity}

\begin{proof}
We start with the board in Figure \ref{gigg6} which has six rows with $n$ squares in the first, second, fifth and sixth rows and $2$ squares in the third and fourth rows. We will count the number of tilings of this board in two different ways and arrive at the identity.

\begin{figure}[htb!]
\includegraphics[width=0.39\textwidth]{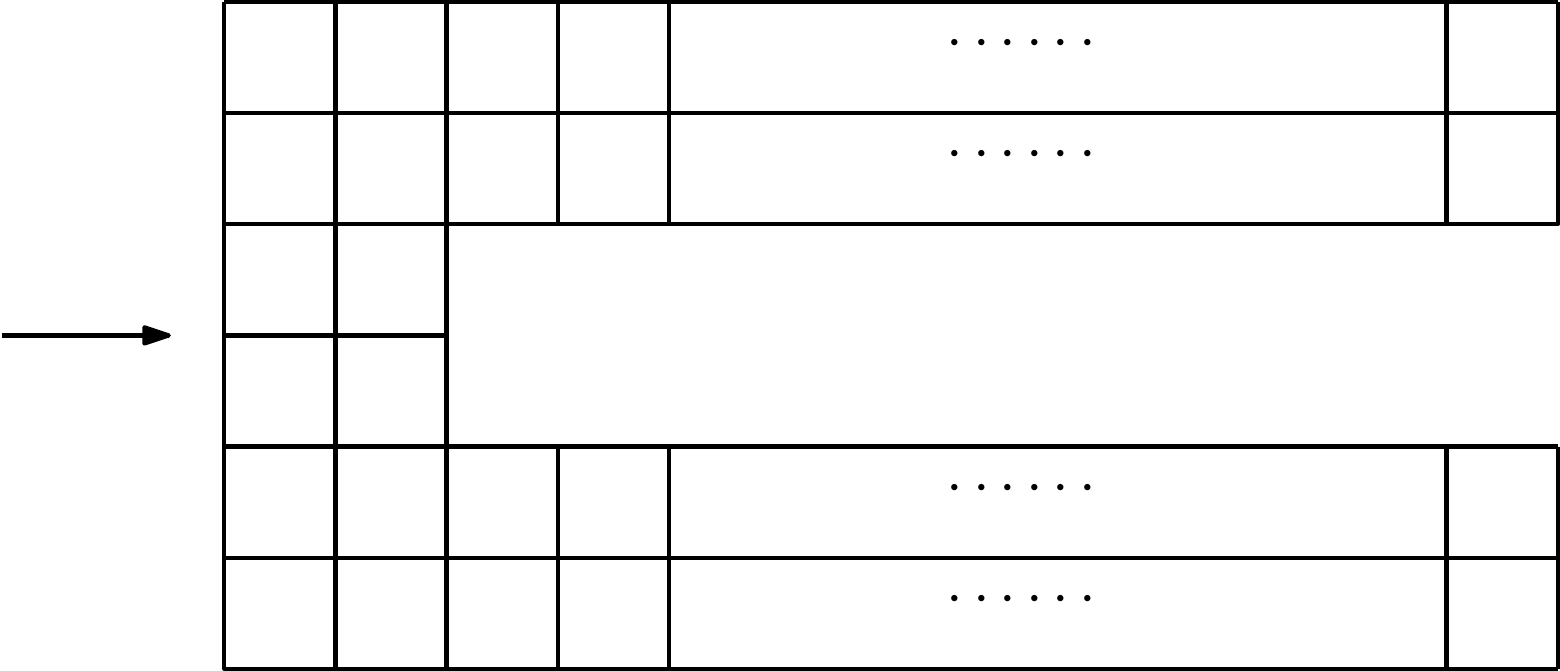}
\caption{Proof of Identity \ref{iden-lucas-1}.} \label{gigg6}
\end{figure}

First, notice that if we can break the board at the arrow then we get in total $L_n^2$ ways of tiling the whole board using dominoes. If we cannot break the board at the arrow then two vertical dominoes covers the squares above and below the arrow and this gives us $f_n^2$ many ways to tile the rest of the board with dominoes. So in total we can tile the whole board in $L_n^2+f_n^2$ many ways using dominoes.

Secondly, we look at the ways of tiling the board from the vertical direction now. Notice that there will be four different cases on which the total number of tilings of the whole board depend upon:
\begin{itemize}
    \item If no horizontal dominoes covers the second and third squares of rows $1, 2, 5$ and $6$.
    \item If a horizontal domino covers the second and third squares of row $1$, which will force another horizontal domino to cover the second and third squares of row $2$.
    \item If a horizontal domino cover the second and third squares of row $5$, which will force another horizontal domino to cover the second and third squares of row $6$.
\item If a horizontal domino covers the second and third squares of rows $1, 2, 5$ and $6$.
\end{itemize}
Counting the number of tilings with dominoes in each of the above cases gives us the right hand side of the identity.
\end{proof}

\begin{identity}\label{iden-lucas-2}
For all $n\geq 2$ and $m\geq 5$ we have,
\begin{equation}\label{eq:gg3}
L_n^2f_{m-6}+2L_nf_nf_{m-7}+f_n^2f_{m-8}=f_{n-2}^2f_m+2f_{n-2}f_{n-3}f_{m-2}+f_{n-3}^2f_{m-4}.
\end{equation}
\end{identity}

\begin{proof}
We start with a board with $m$ rows, where rows $1, 2, m-1$ and $m$ have $n$ squares each and all other rows have $2$ squares each, as in Figure \ref{gigg7}.
\begin{figure}[htb!]
\includegraphics[width=0.5\textwidth]{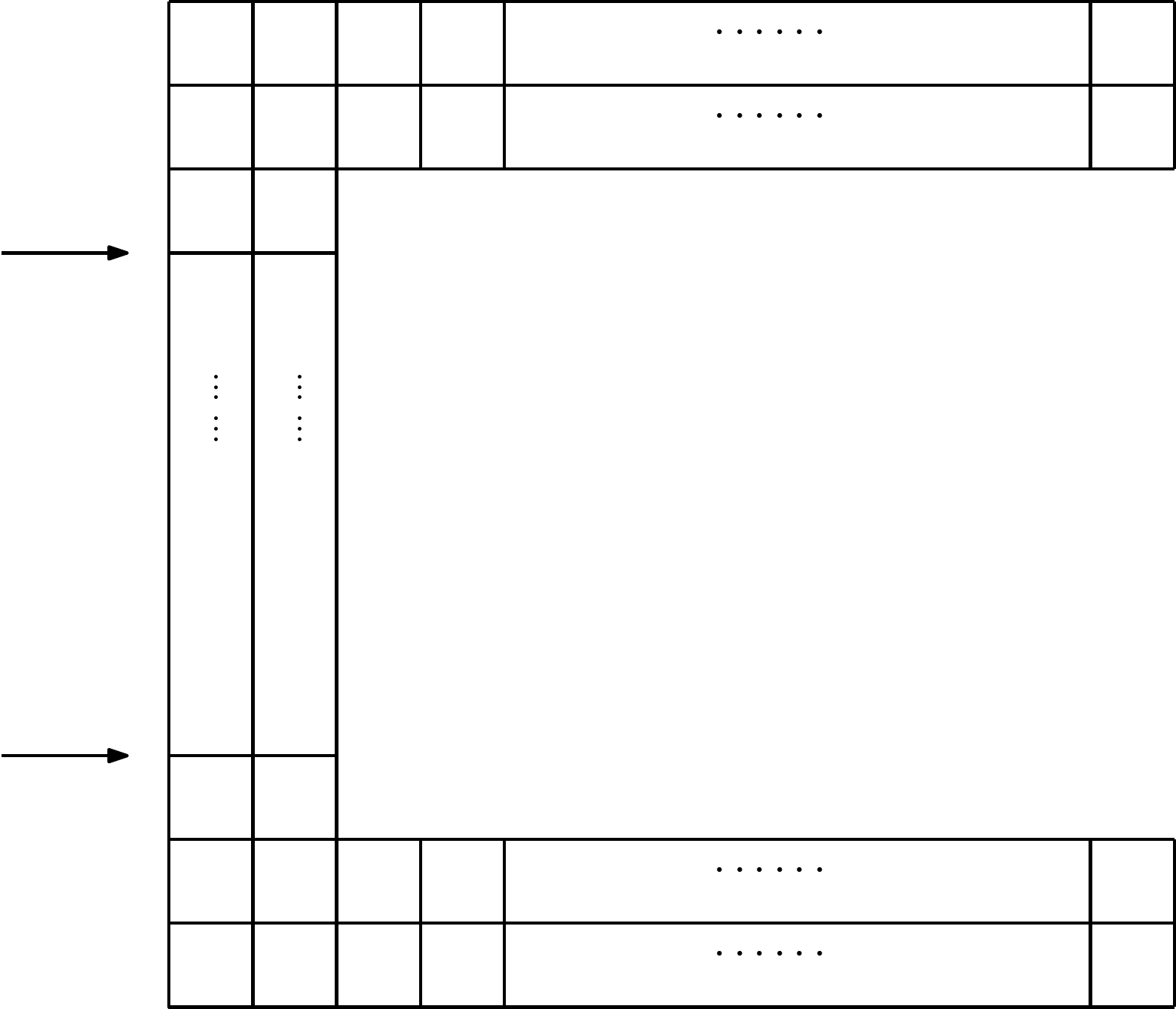}
\caption{Proof of Identity \ref{iden-lucas-2}.} \label{gigg7}
\end{figure}

The right hand side of the identity is easy to derive in a similar way as we did for Identity \ref{iden-lucas-1}.

For the left hand side, instead of looking at whether the board can be broken in one place, we now look at the cases where the board can be broken or not in two place, marked with two arrows here. This will give rise to the following cases:
\begin{itemize}
    \item The board can be broken at the top arrow.
    \item The board can be broken at the bottom arrow.
    \item The board can be broken at both the arrows, into three pieces.
    \item The board cannot be broken at either arrows.
\end{itemize}
Counting the number of ways of tiling the board with dominoes in each of the above cases, like we did for the left hand side of Identity \ref{iden-lucas-1} we shall arrive at our result.
\end{proof}

\begin{remark}
The identity \eqref{eq:gg3} holds for all integral values of $n$ and
$m$, but using the above combinatorial method it can be proved
when $m\geq 5$ and $n\geq 2$.
\end{remark}

\subsection{Combinatorial interpretation of Gibonacci numbers}\label{sec:cGibonacci}

We make a distinction in two cases, first we deal with the case when where $G_0>G_1$ and then we deal with the more general case.

\textbf{Case 1: $G_0>G_1$.} We still use the modified $2 \times n$
board as in Figure \ref{fig1}, but we count them with some
conditions. If a domino covers the squares $x$ and $y$ then we can
assign $G_1$ colours to that domino, and if a domino covers the
square $x$ and the square below it then we can assign $G_0-G_1$
colours to it. All the remaining dominoes are of one colour. A
similar argument as in Subsection \ref{sec:cLucas} will give us
the total number of domino tilings of this board to be
$G_1f_n+(G_0-G_1)f_{n-2}$, which by equation \eqref{fgeq} is $G_n$.

\begin{remark}
This interpretation implies the interpretation of the Lucas
numbers discussed in the previous subsection.
\end{remark}

\textbf{Case 2: General Case.} We use a $2\times n$ board where we mark three squares $x, y$ and $z$ as shown in Figure \ref{fig:gib-g}.We wish to count the number of domino tilings of this board with some conditions. Any domino which covers the squares marked $x$ and $y$ can be assigned $G_0$ colours, any domino which covers the squares marked $x$ and $z$ can be assigned $G_1$ colours, while the
remaining dominoes are of the same colour. With this assignment, it is now not difficult to see that the number of domino tilings of this board is $G_0f_{n-2}+G_1f_{n-1}$, which by equation \eqref{fgeq} is $G_n$.

\begin{figure}[htb!]
\includegraphics[width=0.39\textwidth]{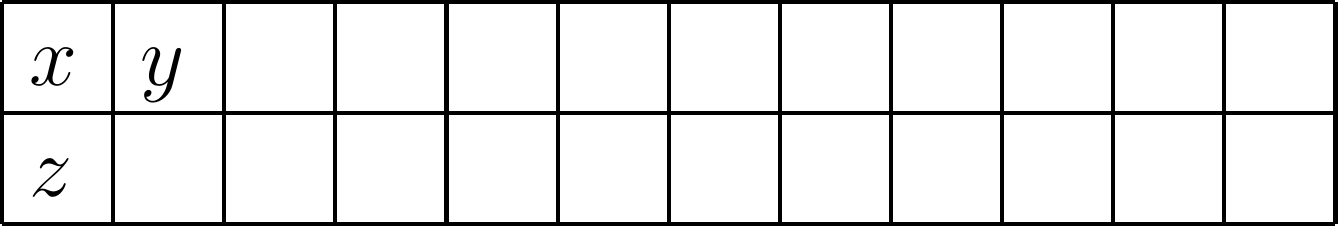}
\caption{$2\times 12$ board with marked squares.}
\label{fig:gib-g}
\end{figure}

We will use this interpretation of Gibonacci numbers in the succeeding sections to prove new identities as well as reprove some known identities.

\begin{remark}
As in the case for Lucas numbers (see Remark \ref{rem:lucas}), we can also get the recurrence for Gibonacci numbers directly from the above interpretation.
\end{remark}

\section{Some new identities involving Gibonacci numbers}\label{sec:identity-new}

Now that we have presented the combinatorial interpretation of Gibonacci numbers, we can use it to prove several new identities. This is done in this section.

\subsection{Identities involving one Gibonacci sequence}

\begin{identity}\label{th1}
For all $n\in \mathbb{N}$ and all $N\leq n$ we have,
\begin{equation}
G_NG_{n-N}+G_{N-1}G_{n-N-1}=G_1G_{n-1}+G_0G_{n-2}.
\end{equation}\label{eq:eq1}
\end{identity}

\begin{proof}
Let us mark a $2\times n$ board as shown in the Figure \ref{fig3}, where we have marked squares $a,b,c,x,y$ and $z$ as well as indicated the column $N$ by an arrow below it.

\begin{figure}[htb!]
\includegraphics[width=0.9\textwidth]{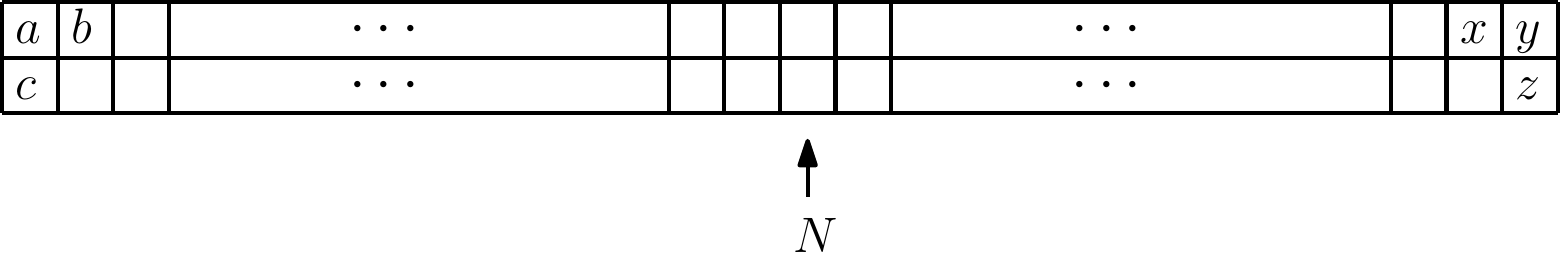}
\caption{$2\times n$ board with marked squares and column $N$ marked.} \label{fig3}
\end{figure}

We wish to count the number of domino tilings of this board with the following conditions:
\begin{itemize}
    \item The squares marked $a$ and $b$ can be covered by $G_0$ colors of dominoes,
    \item The squares marked $x$ and $y$ can be covered by $G_0$ colors of dominoes,
    \item The squares marked $a$ and $c$ can be covered by $G_1$ colors of dominoes,
    \item The squares marked $y$ and $z$ can be covered by $G_1$ colors of dominoes, and
    \item All the remaining dominoes are of same color.
\end{itemize}

Now we will count the number of domino tilings of this board in the following two
ways.

First, let us break the board at column $N$. From columns $1$ to $N$ the board can be tiled in $G_N$ many ways, while from columns $N+1$ to $n$ the board can be tiled in $G_{n-N}$ many ways. So, the total number of such tilings is $G_NG_{n-N}$. Now let two dominoes cover the columns $N$ and $N+1$ horizontally. In this case, the total number of tilings is $G_{N-1}G_{n-N-1}$. Therefore, the total number of tilings of the board is
$G_NG_{n-N}+G_{N-1}G_{n-N-1}$.

Second, let a
vertical domino covers the squares marked $y$ and $z$. Then the number of tilings is
$G_{n-1}G_1$. Now let a horizontal domino cover the squares marked $x$ and $y$. Then the total
tilings is $G_{n-2}G_0$. Therefore, the total tilings of the board is
$G_{n-1}G_1+G_{n-2}G_0$.

This proves our identity.
\end{proof}

\begin{remark}
Notice that the right hand side of the identity \eqref{eq:eq1} is independent of $N$.
\end{remark}

We denote the conditions on the $2\times n$ board in Figure \ref{fig3} by Figure \ref{fig4}. With the help of similar
figures we will prove the rest of identities without going into a detailed explanation of the conditions of colours on dominoes. Sometimes we will use Figure \ref{fig4} in our proofs where we will clearly state the values of $n$ and $N$ that we consider.

\begin{figure}[htb!]
\includegraphics[width=0.9\textwidth]{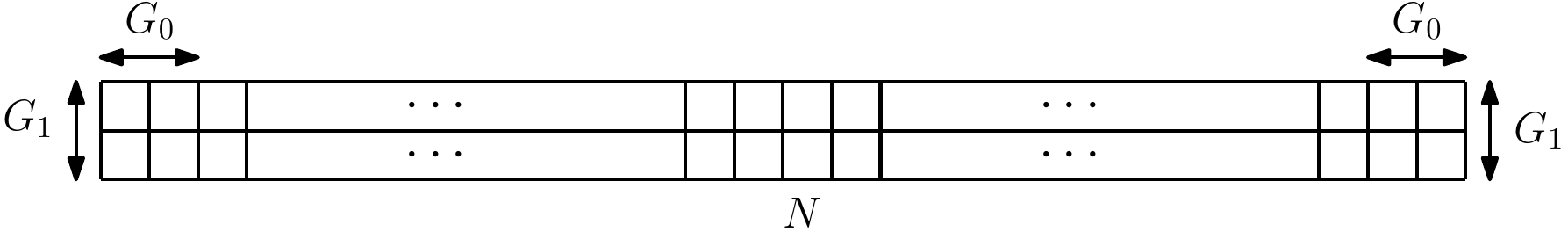}
\caption{Shortened representations of the conditions on an $2\times n$ board.} \label{fig4}
\end{figure}

\begin{corollary}
For all $n\in \mathbb{N}$, we have
\[
G_{n+1}^2+G_n^2=G_0G_{2n}+G_1G_{2n+1}.
\]
\end{corollary}

\begin{proof}
Putting $N=n+1$ and $n=2n+2$ in the identity \eqref{eq:eq1} we
will get this identity.

This identity can also be proven combinatorially using a $2\times (2n+2)$ board.
\end{proof}

\begin{identity}
For all $n\in \mathbb{N}$, we have
\[
G_{n+2}^2-G_n^2=G_0G_{2n+1}+G_1G_{2n+2}.
\]
\end{identity}

\begin{proof}
Taking a $2\times (2n+3)$ board as in Figure \ref{fig4} (we ignore $N$ for the moment), we see that the number of domino tilings of that board is
$G_0G_{2n+1}+G_1G_{2n+2}$.

Now, if we take $N=n+2$ and count the number of domino tilings as we did in the proof of Identity \ref{th1}, we see that there are $G_{n+2}G_{n+1}+G_{n+1}G_n=G_{n+2}G_{n+1}+G_{n+2}G_n-G_{n+2}G_n+G_{n+1}G_n=G_{n+2}^2-G_n^2$ such tilings.

This proves the identity.
\end{proof}

\begin{corollary}[Lucas]
For all $n\in \mathbb{N}$, we have
\[
F_{n+1}^2-F_{n-1}^2=F_{2n}.
\]
\end{corollary}

\begin{proof}
We take $G_0=G_1=1$ to get $f_{n+2}^2-f_n^2=f_{2n+1}+f_{2n+2}=f_{2n+3}$, which implies $f_n^2-f_{n-2}^2=f_{2n-1}$. Therefore, $F_{n+1}^2-F_{n-1}^2=F_{2n}$.
\end{proof}

We exploit a technique of Benjamin, Crouch and Sellers  \cite{BenjaminCrouchSellers} to prove the following result.

\begin{identity}\label{id:sellers-1}
For all $n\geq 2$, we have
\begin{equation}\label{eq:sellers}
    (G_1-G_0)\sum_{j=1}^{n-1}G_{2n-1-2j}+G_0\sum_{j=1}^{n-1}jG_{2n-1-2j}+2G_0G_1+G^2_0(n-2)=G_1G_{2n-2}+G_0G_{2n-3},
\end{equation}
and
\begin{equation}\label{eq:sellers-1}
    (G_1-G_0)\sum_{j=1}^{n-1}G_{2n-2j}+G_0\sum_{j=1}^{n-1}jG_{2n-2j}+G_0^2+(G_1+(n-1)G_0)G_1=G_1G_{2n-1}+G_0G_{2n-2}.
\end{equation}
\end{identity}

\begin{proof}
We prove \eqref{eq:sellers} in details. We start with a $2\times (2n-1)$ board as in Figure \ref{fig4}. Clearly the number of domino tiling of this board is $G_1G_{2n-2}+G_0G_{2n-3}$.

Now we focus on the second vertical domino that can occur in a tiling of this board. The possible positions of this second vertical domino are at positions $2, 4, \ldots, 2n-2$. Now, let $2j$ be the position of this second vertical domino. The right hand side of the board to this second vertical domino can be tiled in $G_{2n-1-2j}$ ways. In the left hand side of this second vertical domino, there is exactly one vertical domino and the rest are all horizontal dominoes. If the first vertical domino occurs in the first column then we can tile the right hand side in $G_1$ ways, otherwise there will be $j-1$ choices for placing this first vertical domino and then we an tile the left hand side in $(j-1)G_o$ ways. Thus, the total number of such tilings is $\sum_{j=1}^{n-1}G_{2n-1-2j}(G_1+(j-1)G_0)$.

Now, let us count the remaining tilings where the second vertical domino does not exist. That is, there is only one vertical domino. We get the following three cases:
\begin{itemize}
    \item The position of the vertical domino is at column $1$. This gives us $G_1G_0$ tilings.
    \item The position of the vertical domino is at column $(2n-1)$. This gives us $G_0G_1$ tilings.
    \item The position of the vertical domino is at one of the odd positions from column $3$ to $(2n-3)$. This gives us $G_0(n-2)G_0$ tilings.
\end{itemize}

Summing over all the above gives us the left hand side of \eqref{eq:sellers}.

To prove \eqref{eq:sellers-1}, we take a $2\times 2n$ board and do the same procedure, this time adding the tiling with no vertical domino which gives us $G_0^2$ tilings, as well as the tiling with the second vertical domino at column $2n$ which gives us $(G_1+(n-1)G_0)G_1$ tilings.
\end{proof}

\begin{remark}
Setting $G_1=G_0=1$ in \eqref{eq:sellers} and \eqref{eq:sellers-1} gives back Benjamin, Crouch and Sellers' \cite{BenjaminCrouchSellers} identities (2) and (3).
\end{remark}

Using techniques from Benjamin, Crouch and Sellers \cite{BenjaminCrouchSellers}, we can generalize Identity \ref{id:sellers-1} by considering the $p$-th vertical domino instead of the second vertical domino. This is done in the following theorem.

\begin{theorem}
For $p\geq 2, m\geq p-1$, we have
\begin{multline*}
    \sum_{\mycom{k=p}{k\equiv p \pmod 2}}^{m-1}\left(G_1\binom{(k+p-4)/2}{p-2}+G_0\binom{(k+p-4)/2}{p-1} \right)G_{m-k}\\+\left(G_1\binom{(m+p-4)/2}{p-2}+G_0\binom{(m+p-4)/2}{p-1}\right)G_1\\
    +\sum_{\mycom{t=0}{t\equiv m \pmod 2}}^{p-1}\left(G_1^2\binom{(m+t-4)/2}{t-2}+2G_0G_1\binom{(m+t-4)/2}{t-1}+G_0^2\binom{(m+t-4)/2}{t} \right)\\ =G_1G_{m-1}+G_0G_{m-2}.
\end{multline*}
The binomial coefficient $\binom{n}{r}$ is zero if $n$ is not an
integer.
\end{theorem}

\begin{proof}
We start with a $2\times m$ board as in Figure \ref{fig4}. The total number of domino tilings of this board is $G_1G_{m-1}+G_0G_{m-2}$.

Let us count the number of domino tilings with respect to the $p$-th vertical domino, if it exists. Let this vertical domino be located in the $k$-th column, where $p\leq k<m$. To the right of this, the board can be tiled in exactly $G_{m-k}$ ways. Before this $p$-th vertical domino we have $p-1$ vertical dominoes and $(k-p)/2$ horizontal dominoes in the first row, provided both $k$ and $p$ have the same parity, else no such tilings exist. So there are total $p-1+(k-p)/2=(k+p-2)/2$ many choices for placing the $p-1$ vertical dominoes. Like before, this will give us $\left(G_1\binom{(k+p-2)/2-1}{p-2}+G_0\binom{(k+p-2)/2-1}{p-1} \right)$ ways to tile the left hand side. This gives us the first summand in the left hand side of the identity. The second summand in the left hand side of the identity counts the number of tilings with $t$ vertical dominoes such that $t\leq p-1$. These tilings have $(m-t)/2$ horizontal dominoes in the first row of the tiling, with $m$ and $t$ being of the same parity. In total we have $(m+t)/2$ choices for the $t$ vertical dominoes and these gives rise to $\left(G_1^2\binom{(m+t)/2-2}{t-2}+2G_0G_1\binom{(m+t)/2-2}{t-1}+G_0^2\binom{(m+t)/2-2}{t} \right)$ ways to tile. The middle term of the LHS exists when the $p$-th vertical domino is located in the $m$-th column and will be non-zero only when $m\equiv p \pmod 2$. This gives us the required identity.
\end{proof}

We can actually use the above proof ideas to combine equations \eqref{eq:sellers} and \eqref{eq:sellers-1} and get the following identity, the proof of which is left to the reader.

\begin{identity}
For all $n\in \mathbb{N}$ we have
\begin{multline}
    (G_1-G_0)\sum_{j=1}^{\lfloor (n-1)/2\rfloor}G_{n-2j}+ G_0\sum_{j=1}^{\lfloor (n-1)/2\rfloor}jG_{n-2j}+ \left(G_1\binom{(n-2)/2}{0}+G_0\binom{(n-2)/2}{1}\right)G_1\\
   +\sum_{\mycom{t=0}{t\equiv n \pmod 2}}^1\left(2G_0G_1t+G_0^2\binom{(n+t-4)/2}{t}\right)=G_1G_{n-1}+G_0G_{n-2}.
\end{multline}
\end{identity}

We can also count with respect to the location of the $p$-th horizontal domino in the first row. This gives us the following result.

\begin{theorem}
For $p\geq 2, m\geq p-1$, we have
\begin{multline*}
    \sum_{k=2p}^m\left(G_1\binom{k-p-2}{p-1}+G_0\binom{k-p-2}{p-2}\right)G_{m-k}\\
    +\sum_{t=0}^{p-1}\left(G_1^2\binom{m-t-2}{t}+2G_0G_1\binom{m-t-2}{t-1}+G_0^2\binom{m-t-2}{t-2} \right)\\=G_1G_{m-1}+G_0G_{m-2}.
\end{multline*}
\end{theorem}

\begin{proof}
Like before the right hand side counts the $2\times m$ board as in Figure \ref{fig4}. Again, let us consider the $p$-th horizontal domino in the first row of such a tiling, if it exists. Let this domino cover the $k-1$ and $k$-th squares in the first row, where $k\geq 2p$. Then, there are $G_{m-k}$ ways to tile the right hand side of the $k$-th column. On the left hand side we have to arrange now $p-1$ horizontal dominoes in the first row and $k-2p$ vertical dominoes. This can be done in $\left(G_1\binom{(k-p-1)-1}{p-1}+G_0\binom{(k-p-1)-1}{p-2}\right)$ ways, which gives us the first summation. The second summation now counts the tilings with $t\leq p-1$  horizontal dominoes in the first row and $m-2t$ vertical dominoes. This can be done in $\left(G_1^2\binom{m-t-2}{t}+2G_0G_1\binom{m-t-2}{t-1}+G_0^2\binom{m-t-2}{t-2} \right)$ ways which gives us the second summation. This proves our identity.
\end{proof}

Instead of taking the board in Figure \ref{fig4}, if we take the board in Figure \ref{fig:gib-g} then we can easily prove the following results in a similar way.

\begin{identity}
For all $n\in \mathbb{N}$, we have
\[
\sum_{j=1}^{n-1}f_{2n-1-2j}(G_1+(j-1)G_0)+G_1+(n-1)G_0=G_{2n-1},
\]
and
\[
\sum_{j=1}^{n-1}f_{2n-2j}(G_1+(j-1)G_0)+G_1+nG_0=G_{2n}.
\]
Thus, we can combine them and write for all $n\in \mathbb{N}$
\[
\sum_{j=1}^{\lfloor (n-1)/2\rfloor}f_{n-2j}(G_1+(j-1)G_0)+G_1+\left\lfloor \frac{n}{2}\right\rfloor G_0=G_n.
\]
\end{identity}

\begin{identity}
For all $n\in \mathbb{N}$, we have
\[
\sum_{j=1}^{n-1}jG_{2n-1-2j}+G_1+(n-1)G_0=G_{2n-1},
\]
and
\[
\sum_{j=1}^{n-1}jG_{2n-2j}+G_1+nG_0=G_{2n}.
\]
Thus, we can combine them and write for all $n\in \mathbb{N}$
\[
\sum_{j=1}^{\lfloor (n-1)/2\rfloor}jG_{n-2j}+G_1+\left\lfloor \frac{n}{2}\right\rfloor G_0=G_n.
\]
\end{identity}

\begin{theorem}
For $p\geq 2, m\geq p-1$, we have
\begin{multline*}
    \sum_{\mycom{k=p}{k\equiv p \pmod 2}}^m\left(G_1\binom{(k+p-4)/2}{p-2}+G_0\binom{(k+p-4)/2}{p-1} \right)f_{m-k}\\
    +\sum_{\mycom{t=0}{t\equiv m \pmod 2}}^{p-1}\left(G_1\binom{(m-t-2)/2}{t-1}+G_0\binom{(m-t-2)/2}{t} \right)=G_m.
\end{multline*}
The binomial coefficient $\binom{n}{r}$ is zero is $n$ is not an
integer.
\end{theorem}

\begin{theorem}
For $p\geq 2, m\geq p-1$, we have
\begin{multline*}
    \sum_{\mycom{k=p}{k\equiv p \pmod 2}}^{m-1}\binom{(k+p-2)/2}{p-1}G_{m-k}+\binom{(m+p-2)/2}{p-1}G_1\\
     +\sum_{\mycom{t=0}{t\equiv m \pmod 2}}^{p-1}\left(G_1\binom{(m-t-2)/2}{t-1}+G_0\binom{(m-t-2)/2}{t} \right)=G_m.
\end{multline*}
The binomial coefficient $\binom{n}{r}$ is zero is $n$ is not an
integer.
\end{theorem}

\begin{theorem}
For $p\geq 2, m\geq p-1$, we have
\begin{multline*}
    \sum_{k=2p}^m\left(G_1\binom{k-p-2}{p-1}+G_0\binom{k-p-2}{p-2}\right)f_{m-k}\\
    +\sum_{t=0}^{p-1}\left(G_1\binom{m-t-1}{t}+G_0\binom{m-t-1}{t-1} \right)=G_m.
\end{multline*}
\end{theorem}

\begin{theorem}
For $p\geq 2, m\geq p-1$, we have
\[
    \sum_{k=2p}^m \binom{k-p-1}{p-1}G_{m-k}
    +\sum_{t=0}^{p-1}\left(G_1\binom{m-t-1}{t}+G_0\binom{m-t-1}{t-1} \right)=G_m.
\]
\end{theorem}

\subsection{Identities involving two Gibonacci sequences}

We now define a `modified' Gibonacci sequence for our next few
results. Let $\{G^\prime_n\}_{n\geq 0}$ be the sequence defined by
$G^\prime_0=G_1, G^\prime_1=G_0$ and \[
G^\prime_n=G^\prime_{n-1}+G^\prime_{n-2}.
\]

\begin{identity}\label{id:gib-2}
For all $n\geq 1$ we have,
\[
G_0(G_n-G^{\prime}_{n-1})=G_1(G^{\prime}_n-G_{n-1}).
\]
\end{identity}

\begin{proof}
We use the $2\times (n+1)$ board as shown in Figure \ref{gigg1}. Notice that the right hand side of this board is slightly different than the ones we have used so far in Figure \ref{fig4}. We count the number of domino tilings of this board in two ways.

If we start from the right side of the board and count the number of tilings as we did in the proof of Identity \ref{th1}, then we see that the total number of tilings is $G_0G_n+G_1G_{n-1}$.

\begin{figure}[htb!]
\includegraphics[width=0.9\textwidth]{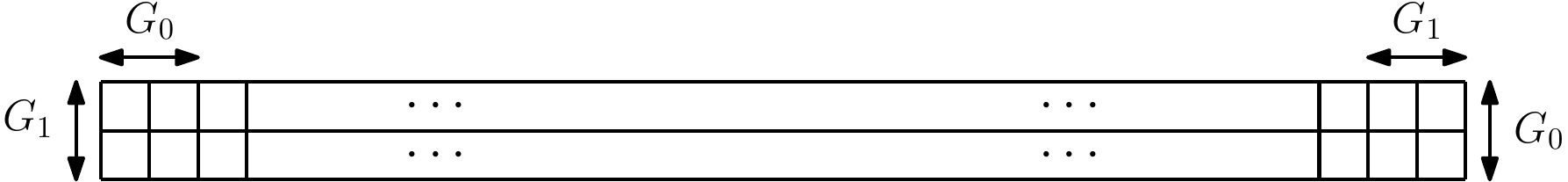}
\caption{$2\times (n+1)$ board associated with Identity \ref{id:gib-2}.} \label{gigg1}
\end{figure}

Alternatively, we start from the left side of the board and count the number of tilings as we did in the proof of Identity \ref{th1}, then we see that the total number of tilings $G_1G^{\prime}_n+G_0G^{\prime}_{n-1}$.

Equating the two numbers and a simple algebraic operation proves the identity.
\end{proof}

\begin{identity}\label{id:th37}
For all $n\geq 2$ and $N\leq n$ we have,
\begin{align*}
     G_NG^{\prime}_{n-N}+G_{N-1}G^{\prime}_{n-N-1}=~& G_0G_{n-1}+G_1G_{n-2}\\
     =~& G_1G^{\prime}_{n-1}+G_0G^{\prime}_{n-2}.
\end{align*}
\end{identity}

\begin{proof}
The proof is similar to the proofs of Identities \ref{th1} and \ref{id:gib-2}. We do not prove it here, but direct the reader to
Figure \ref{gigg2} which proves this.

\begin{figure}[htb!]
\includegraphics[width=0.9\textwidth]{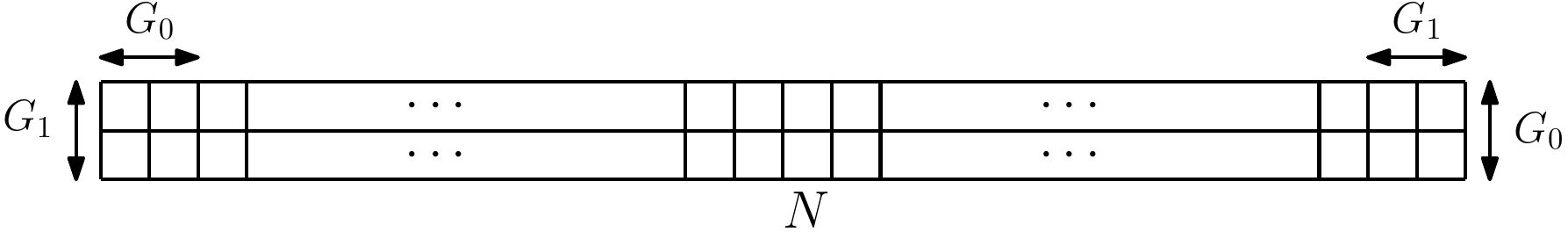}
\caption{$2\times n$ board associated with Identity \ref{id:th37}.} \label{gigg2}
\end{figure}

\end{proof}

\begin{corollary}
For all $n\in \mathbb{N}$, we have
\begin{align*}
     G_nG^{\prime}_n+G_{n-1}G^{\prime}_{n-1}=~& G_0G_{2n-1}+G_1G_{2n-2}\\
     =~& G_1G^{\prime}_{2n-1}+G_0G^{\prime}_{2n-2}.
\end{align*}
\end{corollary}

\begin{proof}
Taking $n=2N$ in Identity \ref{id:th37} and then changing the dummy suffixes proves this.
\end{proof}

\begin{identity}\label{id:gib-fib}
For all $n\geq 3$ we have,
\begin{equation}\label{eq:gg1}
G_n-G^{\prime}_n=(G_1-G_0)f_{n-3}.
\end{equation}
\end{identity}

\begin{proof}
We use the board shown in Figure \ref{gigg3}: we have a $2\times n$ board placed horizontally, and then we extend the first two columns of this board to $n$ rows vertically below. We now count the number of domino tilings of this board in two different ways.

First let us look at the horizontal $2\times n$ board. If no domino cover the second and third squares of the first column, then the total number of tilings is $G_nf_{n-2}$. Otherwise the total number of such tilings is $G_0f_{n-2}f_{n-3}$. So in total we get
$G_nf_{n-2}+G_0f_{n-2}f_{n-3}$ many tilings of the board with dominoes.

\begin{figure}[htb!]
\includegraphics[width=0.5\textwidth]{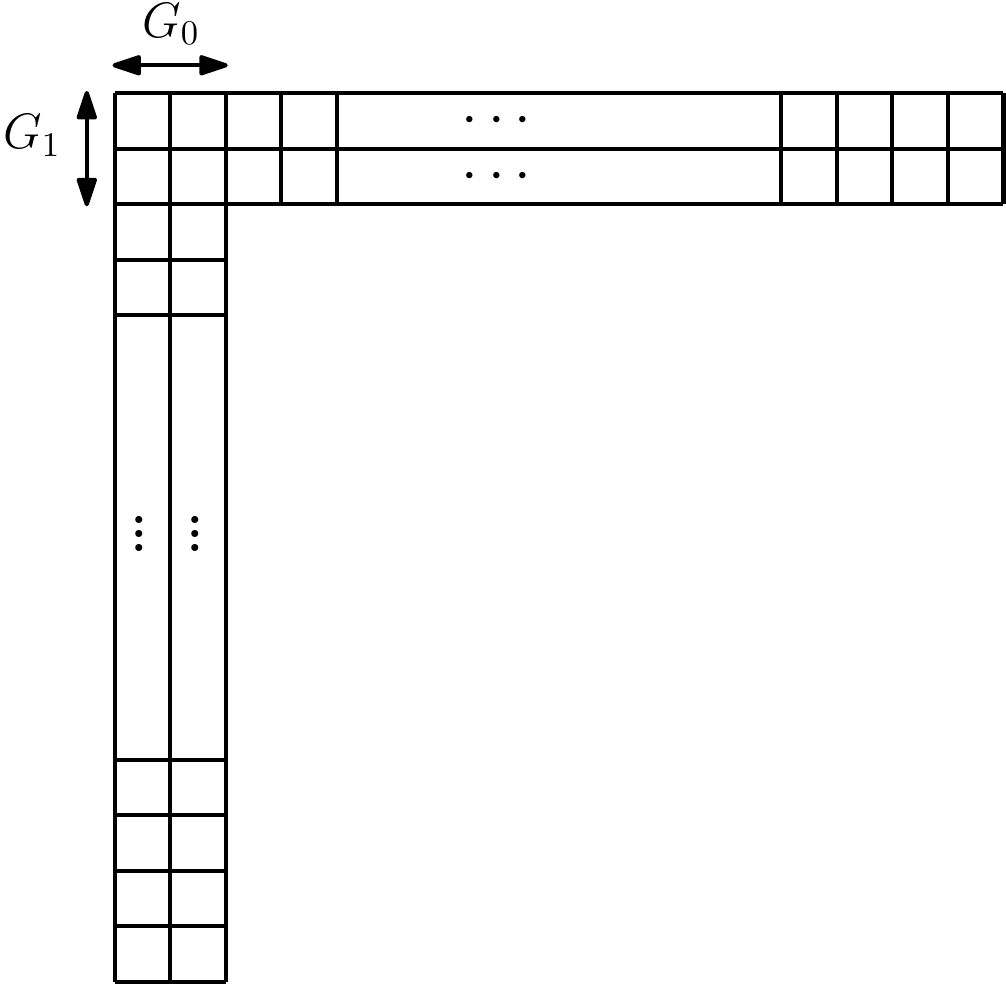}
\caption{Board associated with Identity \ref{id:gib-fib}.} \label{gigg3}
\end{figure}

Now, let us count the number of tilings by looking at the board in the vertical direction. If no domino cover the second and third squares of the first row, then the number of tilings is $G^{\prime}_nf_{n-2}$. Otherwise, the number is $G_1f_{n-2}f_{n-3}$. So, in total the number of domino tilings of the board is $G^{\prime}_nf_{n-2}+G_1f_{n-2}f_{n-3}$.

Equating the two numbers we get $G_n+G_0f_{n-3}=G^{\prime}_n+G_1f_{n-3}$ which proves the identity.
\end{proof}

Two very simple congruence relations follow as corollaries.

\begin{corollary}
For all $n\geq 3$, we have
\[
G_n\equiv G^{\prime}_n \pmod {f_{n-3}}.
\]
\end{corollary}

\begin{corollary}
For all $n\in \mathbb{N}$, we have
\[
G_n\equiv G^{\prime}_n \pmod {G_1-G_0}.
\]
\end{corollary}

\begin{identity}\label{id:gib-fib-2}
For all $n\geq 2$, we have
\begin{equation}\label{eq:gg2}
G_1(G_n-G^{\prime}_n)+G_0(G_{n-1}-G^{\prime}_{n-1})=(G_1-G_0)G_{n-2}.
\end{equation}
\end{identity}

\begin{proof}
We use the board shown in Figure \ref{gigg4}: we have a $2\times n$ board placed horizontally, and then we extend the first two columns of this board to $n$ rows vertically below. We now count the number of domino tilings of this board in two different ways.

First let us look at the horizontal $2\times n$ board. If a domino covers the first and second squares in the first column then the number of domino tilings of the whole board is $G_1G_{n-1}G_{n-2}$. Otherwise, the number of such tilings is $G_0G_{n-2}G_{n-1}$. So the total number of tilings of the board is $G_1G_{n-1}G_{n-2}+G_0G_{n-2}G_{n-1}$.

\begin{figure}[htb!]
\includegraphics[width=0.5\textwidth]{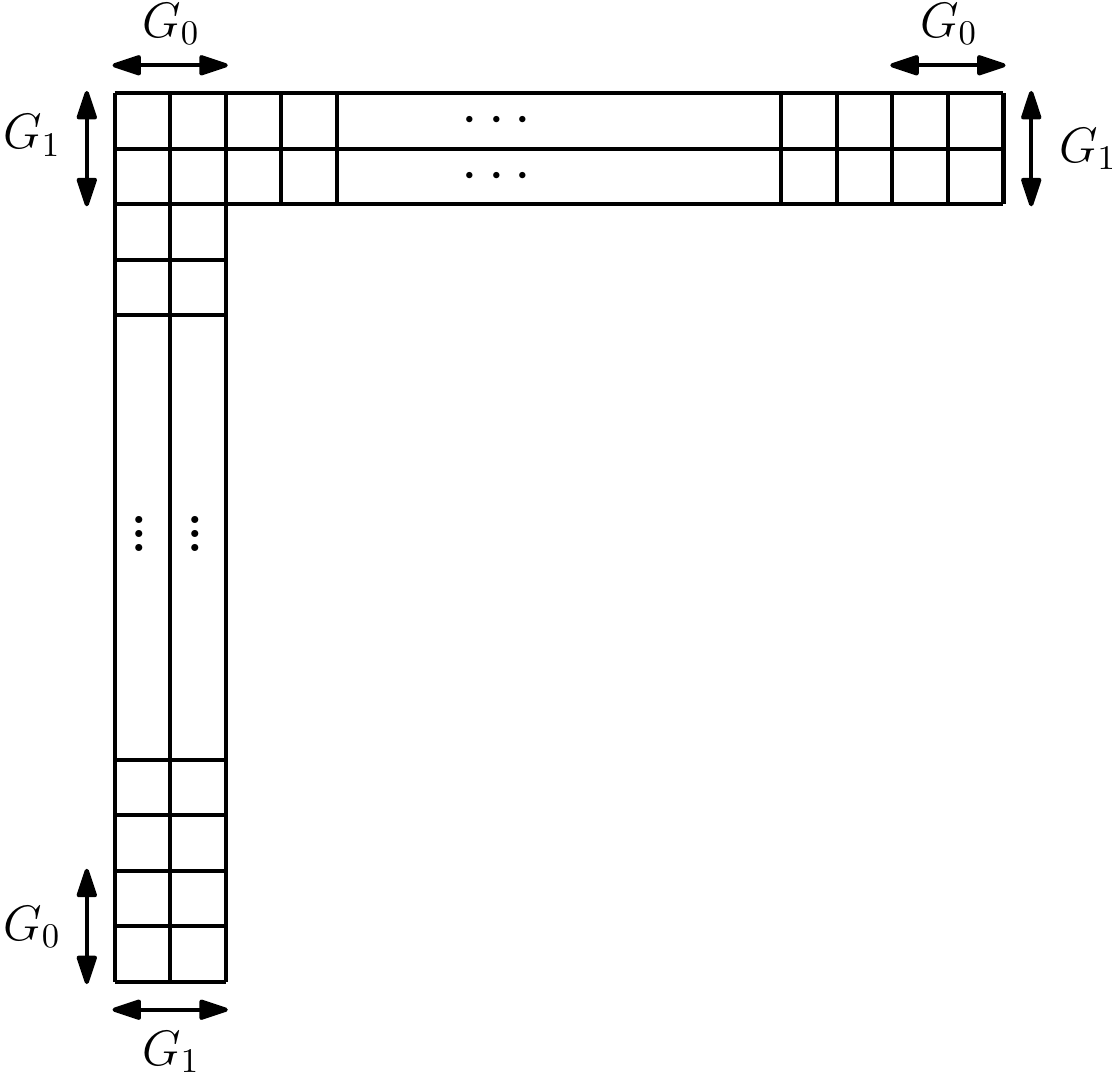}
\caption{Board associated with Identity \ref{id:gib-fib-2}.} \label{gigg4}
\end{figure}

Alternatively, let us now look at the figure in the vertical direction (the reader can think of this as rotated by $90^\circ$). If a domino covers the second and third squares in the first row then the number of tilings in this case is $G_1G_{n-2}G_{n-3}$. Otherwise, we can break the board in between the second and third columns. One part of this broken board will contribute $G_{n-2}$ many tilings. For the other part we look at the bottom most squares of the first and second rows: if a domino covers the bottom most squares in the first and second row then the total number of tilings of the broken part is $G_1G^\prime_{n-1}$, otherwise the number is $G_0G^\prime_{n-2}$. So, in total we get the number of domino tilings of the whole board is $G_1G_{n-2}G_{n-3}+G_{n-2}(G_1G^\prime_{n-1}+G_0G^\prime_{n-2})$.

Equating the two numbers from above and after some algebraic manipulation we get
\[
G_1G_{n-2}(G_{n-1}-G^\prime_{n-1})+G_0G_{n-2}(G_{n-2}-G^\prime_{n-2})=(G_1-G_0)G_{n-2}G_{n-3}.
\]
This proves the identity after substituting $n\rightarrow n+1$.
\end{proof}

\begin{remark}
The identity \eqref{eq:gg2} can be proved algebraically using the
identity \eqref{eq:gg1}.
\end{remark}

\begin{identity}
For all $m,n\geq 3$ we have,
\begin{equation}\label{eq:11}
    G_nf_{m-2}+G_0f_{n-2}f_{m-3}=G^{\prime}_mf_{n-2}+G_1f_{n-3}f_{m-2},
\end{equation}
and
\begin{equation}\label{eq:111}
    (G_1G_{n-1}+G_0G_{n-2})G_{m-2}+G_0G_{n-2}G_{m-3}=(G_1G^{\prime}_{m-1}+G_0G^{\prime}_{m-2})G_{n-2}+G_1G_{n-3}G_{m-2}.
\end{equation}
\end{identity}

\begin{proof}
If we take an $2\times m$ board instead of the $2\times
n$ board in the vertical position in Figure \ref{gigg3}, then we get the identity \eqref{eq:11}. Similarly, from Figure \ref{gigg4}, we get the identity \eqref{eq:111}.
\end{proof}

\begin{identity}\label{id:gib-fib-3}
For all $n\geq 2$, we have
\[
(G_0G_n+G_1G_{n-1}-G_1G^{\prime}_{n-2})G_{n-1}=(G_1G^{\prime}_n+G_0G^{\prime}_{n-1}-G_0G_{n-2})G^{\prime}_{n-1}.
\]
\end{identity}

\begin{proof}
The proof of this identity is similar to the proof of Identity \ref{id:gib-fib-2}, so we omit the details here. The board that we need to take in this case is shown in Figure \ref{gigg5}.

\begin{figure}[htb!]
\includegraphics[width=0.5\textwidth]{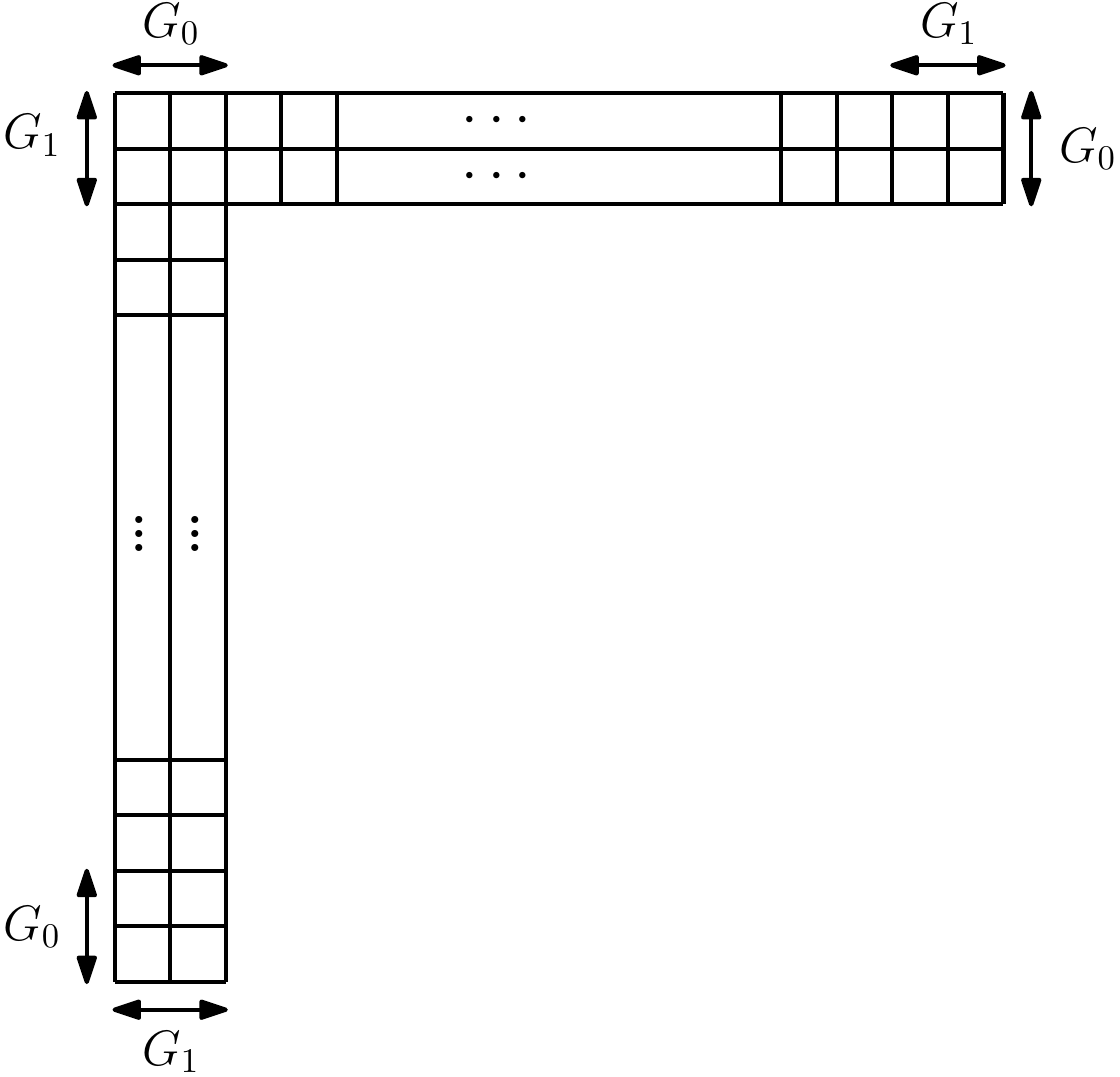}
\caption{Board associated with Identity \ref{id:gib-fib-3}.} \label{gigg5}
\end{figure}

\end{proof}

\section{Proofs of some known and new results from our interpretation of Gibonacci numbers}\label{sec:four}

Using the ideas discussed for Fibonacci numbers by Benjamin and Quinn \cite[Chapter
1]{bq}, we can easily prove different identities involving
Gibonacci numbers with the help of our
 combinatorial
interpretation of Gibonacci numbers. We prove several such known identities in this section, which in most of the cases yields simpler proofs than the ones presented by Benjamin and Quinn \cite{bq}. We also prove one new result in this section (Theorem \ref{thm:new-1}).

Since, we have already proved several of the results in details, for the sake of brevity in this section we do not write all the proofs in details.

\begin{identity}[Identity 39, \cite{bq}]
For all $n\in \mathbb{N}$, we have
\[
G_0+G_1+G_2+\dots+G_n=G_{n+2}-G_1.
\]
\end{identity}

This can be seen by counting the domino tilings of a $2\times(n+2)$ board with respect to the location of the last pair of horizontal dominoes. The proof is similar to the one given by Benjamin and Quinn \cite{bq}.

\begin{identity}[Identity 62, \cite{bq}]
For all $n\in \mathbb{N}$, we have
\[
G_0+G_2+G_4+\dots+G_{2n}=G_{2n+1}-G_{-1}.
\]
\end{identity}

This can be seen by counting the domino tilings of a $2\times(2n+1)$ board with respect to the location of the last vertical domino. This was left as an exercise by Benjamin and Quinn \cite{bq}.

\begin{identity}[Generalization of Identity 7, \cite{bq}]
For all $n\in \mathbb{N}$, we have
\[
3G_n=G_{n+2}+G_{n-2}.
\]
\end{identity}

The proof is analogous to the proof of Benjamin and Quinn's Identity 7 \cite{bq}, so we omit it here.

\begin{identity}[Identity 67, \cite{bq}]
For all $n\in \mathbb{N}$, we have
\[
\sum_{i=0}^nG_i^2=G_nG_{n+1}+G_0(G_0-G_1).
\]
\end{identity}

This is a generalization of Benjamin and Quinn's Identity 9 \cite{bq} and the proof is analogous to it, so we omit it here. This was left as an exercise by Benjamin and Quinn \cite{bq}. They use the technique of \textit{faults and tails} (\cite[page 7]{bq}) to prove Identity 9. We do not recall the technique here, but we use it in the following.

In the following we assume that $\{G_n\}_{n\geq 0}$ and $\{H_n\}_{n\geq 0}$ are two Gibonacci sequences (possibly with different initial conditions). Counting the domino tilings of a $2\times (n+1)$ board for $\{H_n\}_{n\geq 0}$ and a $2\times n$ board for $\{G_n\}_{n\geq 0}$ boards with respect to the location of the last \emph{fault}, we can find the following more general identity.

\begin{identity}[Identity 42, \cite{bq}]\label{id:45}
For all $n\in \mathbb{N}$, we have
\[
\sum_{i=0}^nH_iG_i=
\begin{cases}
    H_{n+1}G_n+(H_0-H_1)G_0, & \hbox{if $n$ is even;} \\
    H_{n+1}G_n+H_0(G_0-G_1), & \hbox{if $n$ is odd.} \\
\end{cases}
\]
\end{identity}
Taking $H_0=G_{-1}$ and $H_1=G_0$, we get the following identity.
\begin{corollary}[Identity 41, \cite{bq}]
For all $n\in \mathbb{N}$, we have
\[
\sum_{i=1}^nG_{i-1}G_i=
\begin{cases}
    G_n^2-G_0^2, & \hbox{if $n$ is even;} \\
    G_n^2-G_1(G_1-G_0), & \hbox{if $n$ is odd.} \\
\end{cases}
\]
\end{corollary}
\noindent This can also be proved easily using a combinatorial method.

The following was left as an exercise by Benjamin and Quinn \cite{bq}. This identity in itself is a generalization of Cassini's identity (Identity 9 of Benjamin and Quinn \cite{bq})
\[
f_n^2=f_{n+1}f_{n-1}+(-1)^n.
\]
They use a technique called \textit{tail swapping}, which we use here to give a brief proof of the result. We do not discuss in detail what is tail swapping and instead refer the reader to Benjamin and Quinn \cite[page 8]{bq}.

\begin{theorem}[Identity 46, \cite{bq}]\label{thm:old-1}
For all $n\in \mathbb{N}$, we have
\[
G_n^2=G_{n+1}G_{n-1}+(-1)^n(G_0G_2-G_1^2).
\]
\end{theorem}

\begin{proof}
We use two sets of boards as shown in Figure \ref{nfig1} and find a one-to-one correspondence between these two sets. Our sets are
\begin{enumerate}
    \item Two $2\times n$ boards, and
    \item One $2\times (n+1)$ board and one $2\times (n-1)$.
\end{enumerate}

Applying the technique of \emph{tail swapping} proves that the tilings of Set $(1)$
is in one-to-one correspondence with the tilings of Set $(2)$ of Figure \ref{nfig1}, when the last \emph{fault} is at the column
$i$, where $2\leq i\leq n$.

\begin{figure}[htb!]
\includegraphics[width=0.9\textwidth]{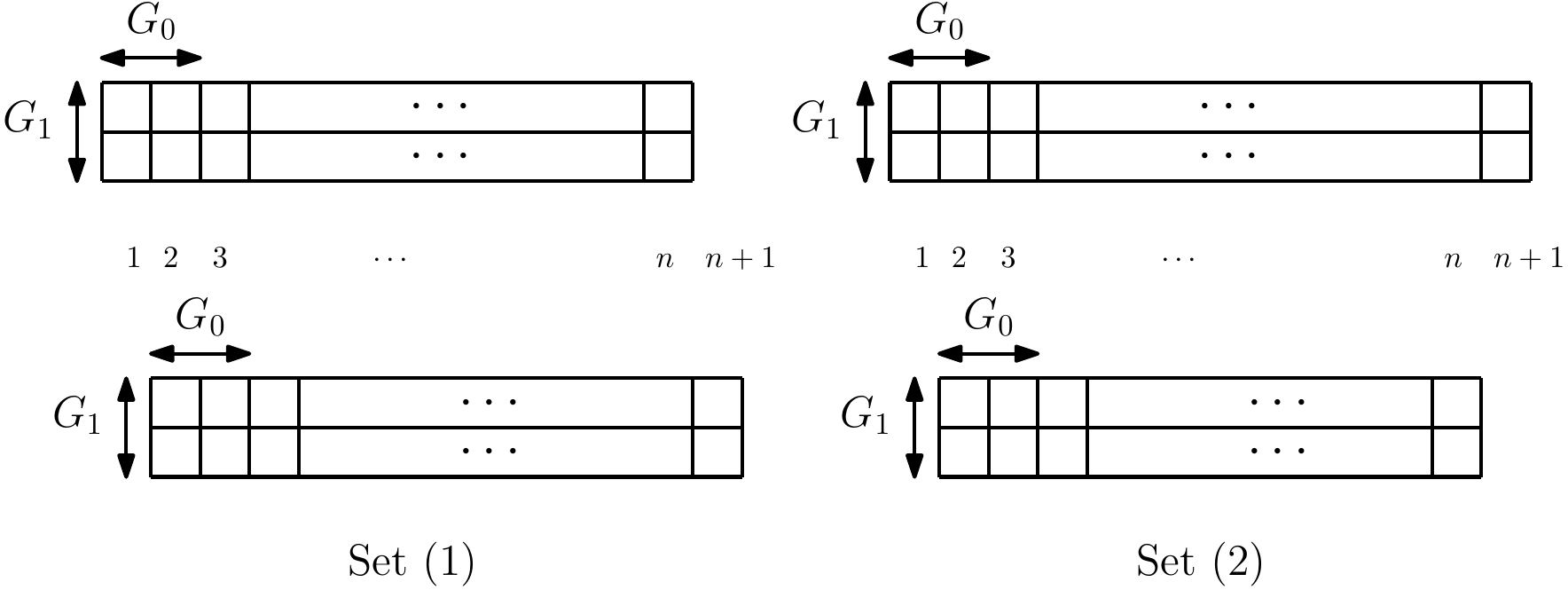}
\caption{Two sets of boards used in the proof of Theorem \ref{thm:old-1}.} \label{nfig1}
\end{figure}

Now let $n$ be odd and the last fault be at column $1$. Then in Set $(1)$ the columns $2$ and $3$ of the top board must occupy
horizontal dominoes; and column $2$ of the bottom board must occupy a vertical domino. So in Set $(1)$ we get in total $G_1^2$ domino tilings. And in Set $(2)$ we get $G_0G_1$ domino tilings. Also, there is no fault-free tiling in Set $(1)$ and there are
$G_0^2$ fault-free tilings in Set $(2)$. Hence $G_n^2-G_1^2=G_{n+1}G_{n-1}-G_0G_1-G_0^2$, which implies $ G_n^2=G_{n+1}G_{n-1}-(G_0G_2-G_1^2)$.

Again, let $n$ be even and the last fault is at column $1$. Then in
Set $(1)$ the column $2$ of the top board must occupy a vertical domino; and columns $2$ and $3$ of the bottom board must
occupy horizontal dominoes. So in Set $(1)$ we get $G_0G_1$ domino tilings. And therefore in
Set $(2)$ we get $G_1^2$ domino tilings. In this case, there are $G_0^2$ fault-free tilings in Set $(1)$ and
there is no fault-free tilings in Set $(2)$. Hence $G_n^2-G_0G_1-G_0^2=G_{n+1}G_{n-1}-G_1^2$, which implies $ G_n^2=G_{n+1}G_{n-1}+(G_0G_2-G_1^2)$.

This proves the result.
\end{proof}

In a similar way we have found the following more general identity.
\begin{theorem}\label{thm:new-1}
For all $n\in \mathbb{N}$ and $p\leq n$, we have
\[
G_n^2=G_{n+p}G_{n-p}+(-1)^{n+p-1}f_{p-1}(G_0G_{p+1}-G_1G_p).
\]
\end{theorem}

\begin{proof}
The tilings of Set $(1)$ as shown in Figure \ref{nfig2} is in one-to-one correspondence with the tilings
of Set $(2)$ as shown in Figure \ref{nfig23}, when the last fault is at the
column $i$, where $p+1\leq i\leq n$. From columns $1$ to $p$ we get $G_p$ tilings. And from columns
$n+1$ to $n+p$ we get $f_{p-1}$ tilings.

\begin{figure}[!htb]
\includegraphics[width=0.9\textwidth]{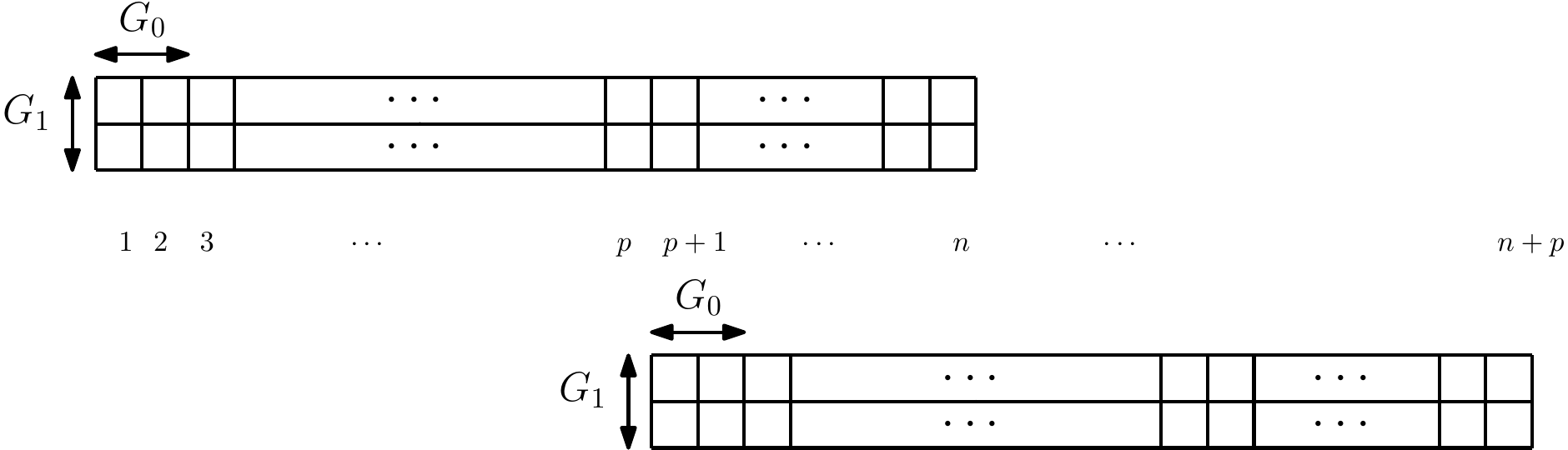}
\caption{Set $(1)$ of boards in the proof of Theorem \ref{thm:new-1}.} \label{nfig2}
\end{figure}

\begin{figure}[!htb]
\includegraphics[width=0.9\textwidth]{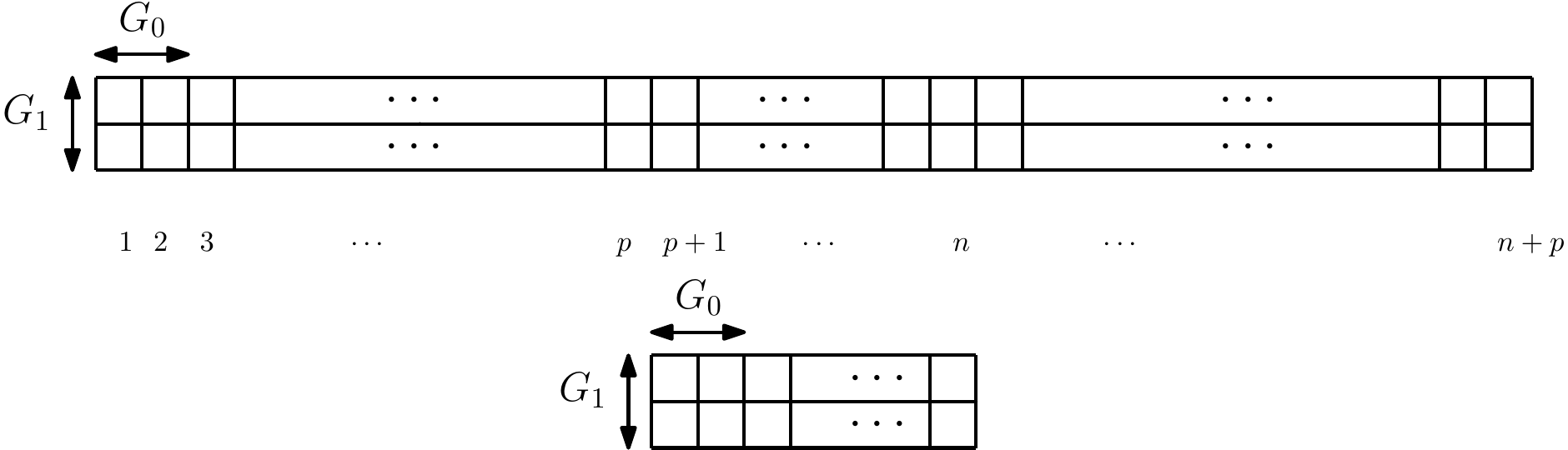}
\caption{Set $(2)$ of boards in the proof of Theorem \ref{thm:new-1}.} \label{nfig23}
\end{figure}

\textbf{Case 1: $n$ and $p$ are both odd.} In this case $n-p$ is even. So, when the last fault is at column $p$, then in Set $(1)$ the
columns $p+1$ and $p+2$ of the top board must occupy
horizontal dominoes; and the column $p+1$ of the bottom board must
occupy a vertical domino. So in the Set $(1)$ we get $G_1G_pf_{p-1}$ tilings. And therefore in the Set $(2)$ we get $G_0G_pf_{p-1}$ tilings. Also, there is no fault-free tilings in Set $(1)$ and there are
$G_0G_{p-1}f_{p-1}$ fault-free tilings in Set $(2)$.

Hence
$G_n^2-G_1G_pf_{p-1}=G_{n+p}G_{n-p}-G_0G_pf_{p-1}-G_0G_{p-1}f_{p-1}$, which implies $$ G_n^2=G_{n+p}G_{n-p}-f_{p-1}(G_0G_{p+1}-G_1G_p).$$

\textbf{Case 2: $n$ is odd and $p$ is even.} In this case $n-p$ is odd. So, when the last fault is at column $p$, then in Set $(1)$ the column $p+1$ of the top board must occupy vertical domino; and the
columns $p+1$ and $p+2$ of the bottom board must occupy horizontal dominoes. So in Set $(1)$ we get $G_0G_pf_{p-1}$ tilings. And therefore in
Set $(2)$ we get $G_1G_pf_{p-1}$ tilings. Also, there are $G_0G_{p-1}f_{p-1}$ fault-free tilings in Set $(1)$
and there is no fault-free tilings in Set $(2)$.

Hence
$G_n^2-G_0G_pf_{p-1}-G_0G_{p-1}f_{p-1}=G_{n+p}G_{n-p}-G_1G_pf_{p-1}$, which implies $$ G_n^2=G_{n+p}G_{n-p}+f_{p-1}(G_0G_{p+1}-G_1G_p).$$

\textbf{Case 3: Remaining cases.} In a similar way when $n$ is even and $p$ is odd implies
$$G_n^2=G_{n+p}G_{n-p}+f_{p-1}(G_0G_{p+1}-G_1G_p)$$ and $p$ is even implies
$$G_n^2=G_{n+p}G_{n-p}-f_{p-1}(G_0G_{p+1}-G_1G_p).$$

This proves the result.
\end{proof}

As corollaries we obtain the following results by taking $p=2$ and $p=n$ respectively.
\begin{corollary}
For all $n\in \mathbb{N}$, we have
\[
G_n^2=G_{n+2}G_{n-2}+(-1)^{n+1}(G_0G_2-G_1^2).
\]
In particular $$f_n^2=f_{n+2}f_{n-2}-(-1)^n.$$
\end{corollary}
\begin{corollary}
For all $n\in \mathbb{N}$, we have
\[
G_n^2=G_{2n}G_0-f_{n-1}(G_0G_{n+1}-G_1G_n).
\]
\end{corollary}

\section{Other results related to Gibonacci numbers}\label{sec:five}

There is a big wealth of literature concerning tiling proofs of Fibonacci and Lucas identities. It would make the present work much longer if we survey all of this literature and apply our techniques to them and generalize the results to Gibonacci numbers. The aim of the present section is to just take three isolated such incidents and use the techniques to prove results about Gibonacci numbers. It appears that all of the three directions below that we take can be generalized much further to get more general and new identities involving Gibonacci numbers.

\subsection{Divisibility Properties of Gibonacci numbers} We have so far not said anything about the divisibility properties of Gibonacci numbers. There is a wealth of results for such divisibility properties of Fibonacci numbers. For Gibonacci numbers, it seems that it is difficult for prove strong divisibility results like those that exist for Fibonacci numbers. However, we can prove weaker results like the following theorem.

\begin{theorem}
For $G_m>1$ we have if $n=mr$ for some $r$, then $G_n\equiv G_{m-1}F_{(r-1)m} \pmod {G_m}$.
\end{theorem}

\begin{proof}
We use the concept of supertiles as defined by Benjamin and Rouse \cite{BenjaminRouse1} to prove this result. Since $G_n=G_{mr}$ for some $r$. Let us now divide the $2\times mr$ board into $r$ segments which we call supertiles, say $S_1, S_2, \ldots, S_r$. This chopping of the board is done to the right of columns numbered $m, 2m , 3m, \ldots, (r-1)m$ as shown in Figure \ref{fig:super}. Such a chopping might result in two horizontal dominoes covering columns $jm$ and $jm+1$ for some $1\leq j \leq r-1$ being split. If this happens we say that such a supertile $S_j$ is open on the right and $S_{j+1}$ is open on the left. Otherwise we say it is closed.

\begin{figure}[!htb]
\includegraphics[width=0.9\textwidth]{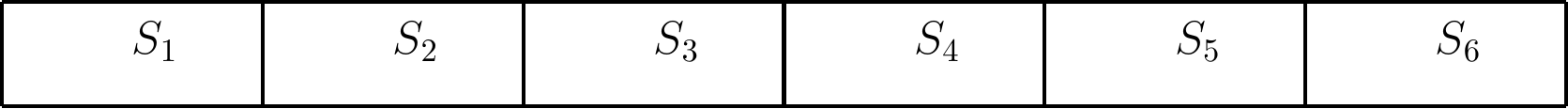}
\caption{Chopping of a $2\times mr$ board into supertiles with $r=6$.} \label{fig:super}
\end{figure}

Let us now look at the first supertile, say $S_j$ that is closed on the left and open on the right. For $1\leq j\leq r$, the number of tilings with $S_j$ being the first of this type is $G_mF_{m+1}^{j-2}F_mF_{(r-j)m}$. Thus we have
\[
G_n=G_{mr}=G_m\sum_{j=2}^{r-1}F_{m+1}^{j-2}F_mF_{(r-j)m}+G_mF_{m+1}^{r-1}+G_{m-1}F_{(r-1)m},
\]
which proves the result.
\end{proof}

Benjamin and Rouse \cite{BenjaminRouse1,BenjaminRouse} prove several more results of similar flavour for Fibonacci numbers (in fact, the congruences are much more strong) as well as generalized Lucas sequences. We do not explore this further here, but it seems that some of the results could no doubt be extended to our setting as well.

\subsection{Breakability of tilings}
The idea of breaking a tiling into two or more parts reveal several interesting identities relating Fibonacci numbers. This aspect has been used quite successfully by Benjamin, Carnes and Cloitre \cite{BenjaminCarnesCloitre} to prove identities involving sums of cubes of Fibonacci numbers. It seems possible to extend their results to prove analogous results for Gibonacci numbers, but we do not proceed in that direction here and leave it as an open problem for the reader. To give a flavour of the type of results they can prove using the concept of breakability of tilings we prove the following simple identity.

\begin{identity}
For all $n\geq 0$ we have,
\[
2\sum_{j=0}^n G_{3j+2}=G_{3n+4}-G_1
\]
\end{identity}
\begin{proof}
There are $G_1$ tilings of a $2\times (3n+4)$ board which cannot be broken at any column of the form $3j+2$ and that is the tiling which starts with a vertical domino and then has $n+1$ sequence of consecutive horizontal dominoes and a vertical domino. So, the right hand side counts the tilings of this board which can be broken at at least one column of the form $3j+2$. The left hand side is now counting the same thing by focusing on the last breakable column of the form $3j+2$, the factor of $2$ comes from the fact that the columns $3j+3$ and $4j+4$ can now be either tiled using two vertical dominoes or two horizontal dominoes.
\end{proof}

\subsection{Lacunary recurrence of Gibonacci numbers}

A recurrence relation involving terms of a given sequence with indices in arithmetic progression is called a lacunary recurrence. Recently, Ballantine and Merca \cite{bm} found such a lacunary recurrence for Fibonacci numbers, while the present authors \cite{ms} found one for the Lucas numbers. We prove the the result of Ballantine and Merca \cite{bm} using our combinatorial interpretation of Gibonacci numbers.

\begin{theorem}\cite[Theorem 1]{bm}\label{thm:merca}
Given a positive integer $N\geq 2$, we have
\[
F_n=F_N\cdot F_{N-1}^{\lfloor \frac{n-1}{N}\rfloor +1}\cdot F_{(n-1)~\text{mod}~N}+F_{N+1}\cdot F_{n-N}+F_N^2\cdot \sum_{k=2}^{\lfloor {\frac{n-1}{N} \rfloor}}F_{N-1}^{k-2}\cdot F_{n-kN},
\]
for all $n\geq N$.
\end{theorem}

\begin{proof}
Using a $2\times n$ board it is easy to see that $G_n=G_{n-N}f_N+G_{n-N-1}f_{n-1}$, which is a known identity (Identity 38 of Benjamin and Quinn \cite{bq}). Again we can show $G_{n-N-1}=G_{n-2N}f_{N-1}+G_{n-2N-1}f_{n-2}$ and $G_{n-2N-1}=G_{n-3N}f_{N-1}+G_{n-3N-1}f_{n-2}$ and so on. Doing
some simple calculations we get
\[
G_n=G_{n-N}f_N+f_{N-1}^2\sum_{i=2}^dG_{n-iN}f_{N-2}^{i-2}+f_{N-1}G_{n-dN-1}f_{N-2}^{d-1},
\]
where $d=\lfloor\frac{n}{N}-1\rfloor+1$.

In particular we have
\[
f_n=f_{n-N}f_N+f_{N-1}^2\sum_{i=2}^df_{n-iN}f_{N-2}^{i-2}+f_{N-1}f_{n-dN-1}f_{N-2}^{d-1}.
\]
Therefore
\[
F_{n+1}=F_{n-N+1}F_{N+1}+F_N^2\sum_{i=2}^dF_{n-iN+1}F_{N-1}^{i-2}+F_NF_{n-dN}F_{N-1}^{d-1}.
\]
Taking $n\rightarrow n-1$ we get
\[
F_n=F_{n-N}F_{N+1}+F_N^2\sum_{i=2}^dF_{n-iN}F_{N-1}^{i-2}+F_NF_{n-dN-1}F_{N-1}^{d-1}
\]
where $d=\lfloor \frac{n-1}{N}-1\rfloor +1=\lfloor \frac{n-1}{N}\rfloor $. Hence
\[
F_n=F_{n-N}F_{N+1}+F_N^2\sum_{i=2}^dF_{n-iN}F_{N-1}^{i-2}+F_NF_{(n-1)\mod N}F_{N-1}^{d-1},
\]
where $d=\lfloor \frac{n-1}{N}\rfloor $.
\end{proof}

We have stopped short of proving a more general lacunary recurrence involving only Gibonacci numbers, but we believe that this might be possible to prove using some of our techniques. We leave this as an open problem.

\section{Concluding Remarks}\label{sec:six}

\begin{enumerate}
    \item As we have seen already, several techniques available in the literature can be modified to use with our representation of Gibonacci numbers. A systematic study of all such methods used to prove Fibonacci identities would no doubt yield many new identities.
    \item Can we use the tiling interpretation of Gibonacci numbers to prove inequalities? Or to find a $t$ in terms of $a,b,n$, for which $f_{t-1}\leq
G_n^{a,b}<f_t$? \item Let us denote by
$G_n^{a,b}=af_{n-2}+bf_{n-1}, n\geq 1$, where $a>0,b>0$.
For any positive integer there exist at least one Gibonacci
sequence where this integer appears. For any positive integer $t$
can we find all the Gibonacci sequences where $t$ appears? That
is, can we find the solutions $(a,b,n)$ of the following equation
for any fix $t$, $G_n^{a,b}=t$? \item Can we say something for the
equation $G_n\equiv G_{n+x} \pmod p$, where $p$ is a prime? Can we
find a significant relation between $x$ and $p$? For example, we
have $G_n\equiv G_{n+3} \pmod 2$, $G_n\equiv G_{n+8} \pmod 3$,
$G_n\equiv G_{n+20} \pmod 5$, $G_n\equiv G_{n+16} \pmod 7$,
$G_n\equiv G_{n+10} \pmod {11}$, $G_n\equiv G_{n+14} \pmod {29}$,
etc. Results of these type for Fibonacci and generalized Fibonacci
type sequences were derived using non-combinatorial techniques by Laugier and the second author \cite{kmj}.
    \item Another aspect which is immediately clear is that we can extend some of our proof techniques to three or more term recurrences. Work in this direction will be reported in a forthcoming paper.
\end{enumerate}

\section*{Acknowledgements}
The second author is partially supported by the Leverhulme Trust Research Project Grant RPG-2019-083.

\bibliographystyle{alpha}

\begin{thebibliography}{BCC09}

\bibitem[BCC09]{BenjaminCarnesCloitre}
Arthur~T. Benjamin, Timothy~A. Carnes, and Benoit Cloitre.
\newblock Recounting the sums of cubes of {F}ibonacci numbers.
\newblock In {\em Proceedings of the {E}leventh {I}nternational {C}onference on
  {F}ibonacci {N}umbers and their {A}pplications}, volume 194, pages 45--51,
  2009.

\bibitem[BCS19]{BenjaminCrouchSellers}
Arthur~T. Benjamin, Joshua Crouch, and James~A. Sellers.
\newblock Unified tiling proofs of a family of {F}ibonacci identities.
\newblock {\em Fibonacci Quart.}, 57(1):29--31, 2019.

\bibitem[BM19]{bm}
Cristina Ballantine and Mircea Merca.
\newblock A family of lacunary recurrences for {F}ibonacci numbers.
\newblock {\em Miskolc Math. Notes}, 20(2):767--772, 2019.

\bibitem[BQ03]{bq}
Arthur~T. Benjamin and Jennifer~J. Quinn.
\newblock {\em Proofs that really count}, volume~27 of {\em The Dolciani
  Mathematical Expositions}.
\newblock Mathematical Association of America, Washington, DC, 2003.
\newblock The art of combinatorial proof.

\bibitem[BR04]{BenjaminRouse1}
Arthur~T. Benjamin and Jeremy~A. Rouse.
\newblock Recounting binomial {F}ibonacci identities.
\newblock In {\em Applications of {F}ibonacci numbers. {V}ol. 9}, pages 25--28.
  Kluwer Acad. Publ., Dordrecht, 2004.

\bibitem[BR09]{BenjaminRouse}
Arthur~T. Benjamin and Jeremy~A. Rouse.
\newblock When does {$F^L_m$} divide {$F_n$}? {A} combinatorial solution.
\newblock In {\em Proceedings of the {E}leventh {I}nternational {C}onference on
  {F}ibonacci {N}umbers and their {A}pplications}, volume 194, pages 53--58,
  2009.

\bibitem[Hon85]{hons}
Ross Honsberger.
\newblock {\em Mathematical gems. {III}}, volume~9 of {\em The Dolciani
  Mathematical Expositions}.
\newblock Mathematical Association of America, Washington, DC, 1985.

\bibitem[Kos18]{koshy}
Thomas Koshy.
\newblock {\em Fibonacci and {L}ucas numbers with applications. {V}ol. 1}.
\newblock Pure and Applied Mathematics (Hoboken). John Wiley \& Sons, Inc.,
  Hoboken, NJ, 2018.
\newblock Second edition of [ MR1855020].

\bibitem[LS17]{kmj}
Alexandre {Laugier} and Manjil~P. {Saikia}.
\newblock {Some properties of Fibonacci numbers, generalized Fibonacci numbers
  and generalized Fibonacci polynomial sequences.}
\newblock {\em {Kyungpook Math. J.}}, 57(1):1--84, 2017.

\bibitem[MS20]{ms}
Pankaj~Jyoti Mahanta and Manjil~P. Saikia.
\newblock A family of lacunary recurrences for {L}ucas numbers.
\newblock {\em Fibonacci Quart.}, accepted, 2020.

\end{thebibliography}

\end{document}